\documentclass[CEJM,PDF]{cej} 
\usepackage{layout}

\title{Dynamical  Sieve of Eratosthenes}

\articletype{Research} 

\author{Luis A. Mateos\email{lamateos@amcomputersystems.com}
       }

\institute{AM Computer Systems Research, \\
           1070 Vienna, Austria
          }

\abstract{

In this document, prime numbers are related as functions over time, mimicking the Sieve of Eratosthenes. 
For this purpose, the mathematical representation is a uni-dimentional time line depicting the number line for positive natural numbers $N$, where each number $n$ represents a time $t$. 
In the same way as the Eratosthenes' sieve, which iteratively mark as composite the multiples of each prime, starting at each prime.
This dynamical prime number function P($s$) zero-cross all composite numbers departing from primes, following a linear progression over time.

Moreover, this dynamical prime number function is then modified to zero-cross only $odd$ composite numbers from $odd$ primes, in order to attack the weak Goldbach  conjecture in a non-conventional way, which do not rely directly in trying to add prime numbers. Instead, the main goal is to depict the set of $odd$ numbers bigger than one, with the set of $odd$ primes. Thus, by representing the set of $odd$ numbers bigger than one, every combinatorial sum of three $odd$ prime numbers will result in every $odd$ number bigger than 7.

}


\begin{document}
\maketitle
\section{Introduction }


The sieve of Eratosthenes is a simple algorithm for finding all prime numbers up to any given limit. It iteratively mark as composite the multiples of each prime, starting from that prime \cite{sieves}.

Basically, this sieve algorithm requires the list of numbers upto a limit $z$, in order to mark as $composite$ all prime multiples, starting from $p=2$, and leaving without mark prime numbers.
Likewise, this list structure, in our dynamical representation, is arranged as a time line, describing points as times over a one-dimensional line uniformly separated, starting from zero and ending at an infinite time. Thus, each time $t$ is depicted as a number $n$, correlating both variables, $n = t$.
In order to develop the dynamical Eratosthenes' sieve, the algorithm must become dynamical, a function over time from its initial prime values. Hence, the dynamical sieve must include a starting point and a periodic function, mimicking the iteratively marking of composite numbers by instantaneously zero-cross all multiples of primes, starting from prime times.

\begin{figure}[htp]
\caption{Time line represents the sieve's list of number. \label{fig_1}}
\includegraphics[width=0.5\textwidth]{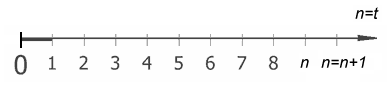}
\end{figure}

Consequently, the dynamical function includes two terms. The first term, $p$, represents the starting point of each prime function, at prime times. The second term includes the $sine$ periodic function, $sin(1/ p)$, representing the continuation in time of the prime periods $(1/ p)$.

\section{Dynamical Prime Number Function P($s$)}

Prime numbers are the building blocks of the positive integers, this was shown by Euler in his proof of the Euler product formula for the Riemann zeta function, Euler came up with a version of the sieve of Eratosthenes, better in the sense that each number was eliminated exactly once \cite{eulerproduct1_cite}. Unlike Eratosthenes' sieve which strikes off multiples of primes it finds from the same sequence, Euler's sieve works on sequences progressively culled from multiples of the preceding primes \cite{eulerproduct2_cite}.


\begin{equation}
\sum_{n=1}^{\infty} \frac{1}{n^s} = \prod_{p \, \mathrm{ prime}} \frac{1}{1-p^{-s}}
\end{equation}

However, in the proposed dynamical sieve, due to the inclusion of a periodic function, the sieve works instantaneously from its starting primes $p$  to $\infty$.
And by following a linear progression, prime numbers are decoded, while composite numbers are encoded over time by this dynamical prime number function $P(s)$ \cite{mateos1_cite}.

\begin{equation}
\label{psfunction}
 P(s) = \sum_{p \, \mathrm{ prime}} \left[ {p^s} + sin \left(\frac{1}{p^s}\right) \right]
\end{equation}


At each time $n^s>1$, for $s>0$, the function P($s$) evaluates if the number in consideration has been zero-cross or not. If a number $n^s$ has not been zero-cross, then the number is a prime $p^s$, this process we called $prime$ $decoding$. While, if a number $n^s$ has been zero-cross, then the number is $composite$, meaning that a previous prime or primes $p^s$ has been decoded and this zero-cross is a multiple of such prime or primes, this process we called $composite$ $encoding$.
Once a prime number $p^s$ is decoded, the function P($s$) start the expansion of the prime $p^s$ into its multiples, zero-crossing $composite$ numbers and decoding primes by leaving intact the time line at prime times $p^s$, as shown in figure \ref{fig_progression}.

\begin{figure}[htp]
\caption{Zero-cross numbers by P(s) =$ \sum_{p} \left[ {p^s} + sin \left(\frac{1}{p^s}\right) \right]$. 
In red $P(s)$ for $p^s=2^s$, in black $P(s)$ for $p^s>2^s$ (odd primes);
white squares zero-cross represent $composite$ numbers; white squares zero-cross by only $P(s)$ for $p^s=2^s$ (red color) represents numbers power of two $2^{s m}$, for $m>1$. \label{fig_progression}}
\includegraphics[width=1\textwidth]{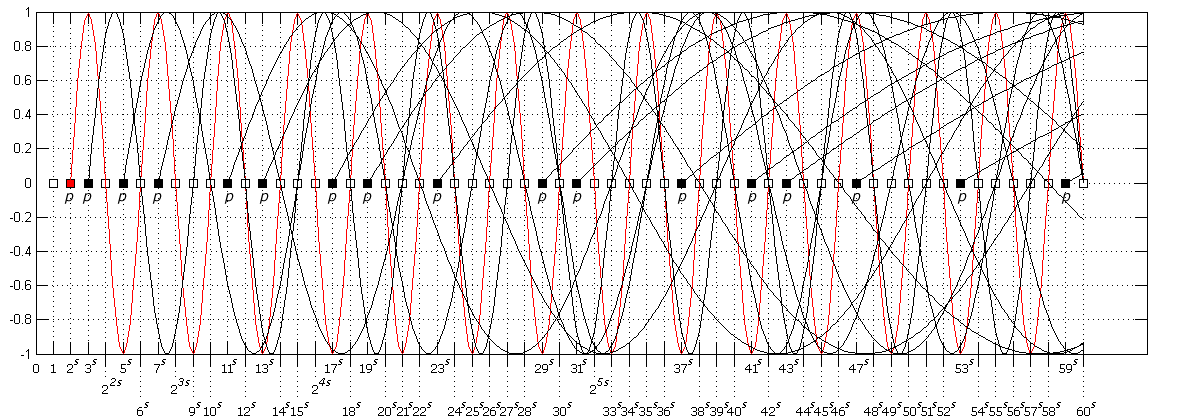}
\end{figure}

The process is as follows:

Starting at time $p^s=2^s$ the function P($s$)$=2^s+ sin \left(\frac{1}{2^s}\right)$ zero-crosses all multiples of $2^s$, while the maximum value of the function exist at all $odd$ times, $n^s_{odd}>1$, as shown in figure \ref{fig_2}.

\begin{figure}[htp]
\caption{P($s$)=$2^s+ sin \left(\frac{1}{2^s}\right)$ start at $p^s=2^s$ and zero-cross multiples of $2^s$. In red $p^s=2^s$. \label{fig_2}}
\includegraphics[width=1\textwidth]{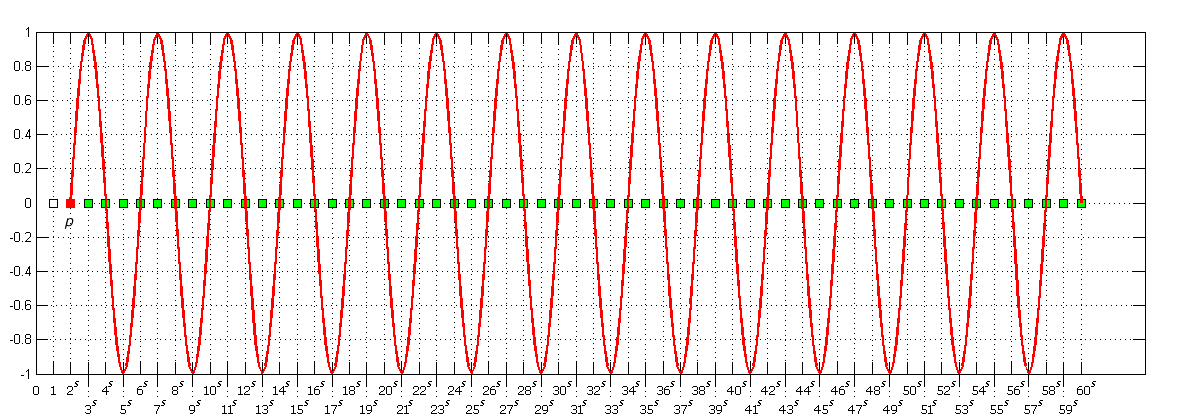}
\end{figure}

At time $p^s=3^s$ no zero-cross occur. Therefore, the function P($s$)$=3^s+ sin \left(\frac{1}{3^s}\right)$ start zero-crossing $odd$ and $even$ multiples of $3^s$, as shown in figure \ref{fig_3}.

\begin{figure}[htp]
\caption{P($s$)=$ \left[ 2^s+ sin \left(\frac{1}{2^s}\right)\right] + \left[ 3^s+ sin \left(\frac{1}{3^s}\right) \right]$ zero-cross multiples of $2^s$ and $3^s$. In red $p^s=3^s$. \label{fig_3}}
\includegraphics[width=1\textwidth]{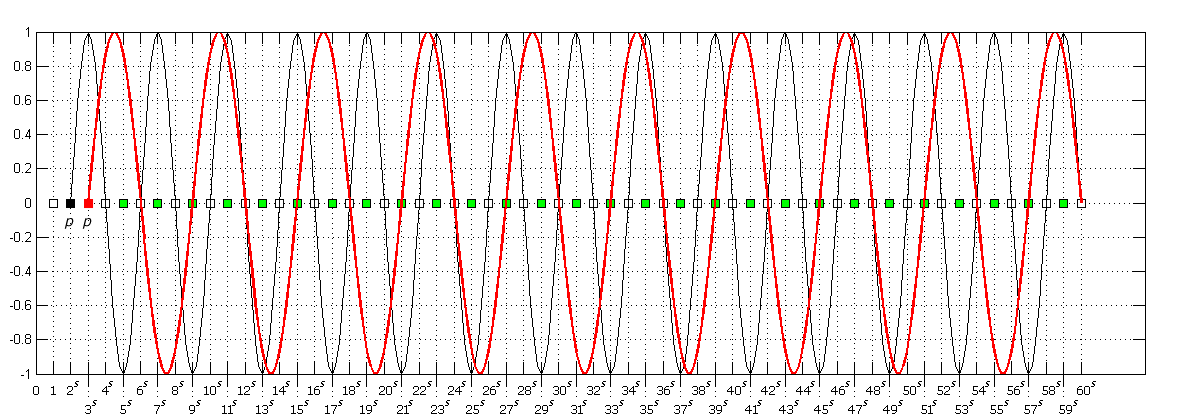}
\end{figure}

At time $n^s=4^s$ $n^s=6^s$,  $n^s=8^s$, $n^s=10^s$ and all $even$ times $n^s>2^s$ a zero-cross occur from the function P($s$)$=2^s+ sin \left(\frac{1}{2^s}\right)$ meaning $composite$ numbers, as shown in figure \ref{fig_2}.

At time  $p^s=5^s$ no zero-cross occur. Therefore, the function P($s$)$=5^s+ sin \left(\frac{1}{5^s}\right)$ start to zero-cross $odd$ and $even$ multiples of $5^s$ , as shown in figure \ref{fig_5}.

\begin{figure}[htp]
\caption{P($s$)=$ \left[ 2^s+ sin \left(\frac{1}{2^s}\right)\right] + \left[ 3^s+ sin \left(\frac{1}{3^s}\right) \right]+ \left[ 5^s+ sin \left(\frac{1}{5^s}\right) \right]$ zero-cross multiples of $2^s$, $3^s$ and $5^s$. In red $p^s=5^s$. \label{fig_5}}
\includegraphics[width=1\textwidth]{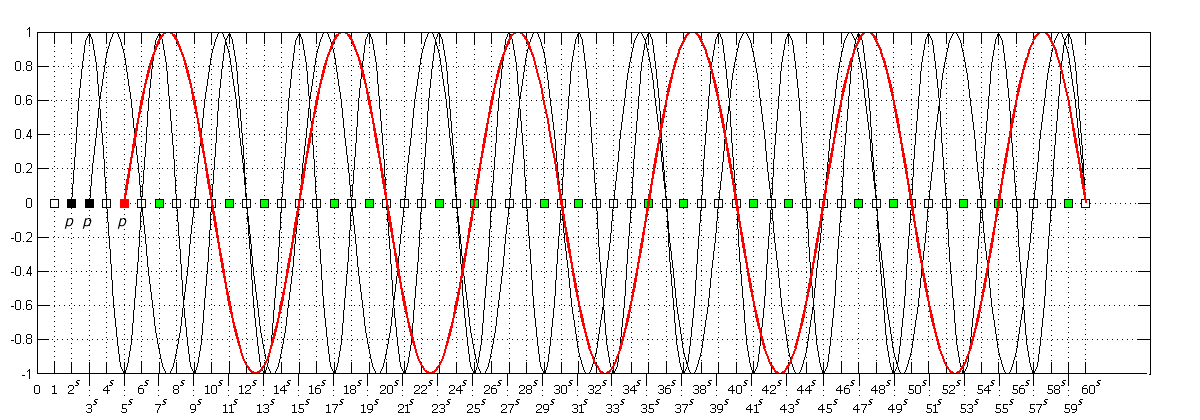}
\end{figure}

At time $p^s=7^s$ no zero-cross occur. Thus, the function P($s$)$=7^s+ sin \left(\frac{1}{7^s}\right)$ start to zero-cross $odd$ and $even$ multiples of $7^s$, as shown in figure \ref{fig_7}.

\begin{figure}[htp]
\caption{P($s$)=$ \left[ 2^s+ sin \left(\frac{1}{2^s}\right)\right] + \left[ 3^s+ sin \left(\frac{1}{3^s}\right) \right]+ \left[ 5^s+ sin \left(\frac{1}{5^s}\right) \right]+ \left[ 7^s+ sin \left(\frac{1}{7^s}\right) \right]$ zero-cross multiples of $2^s$, $3^s$, $5^s$ and $7^s$. In red $p^s=7^s$. \label{fig_7}}
\includegraphics[width=1\textwidth]{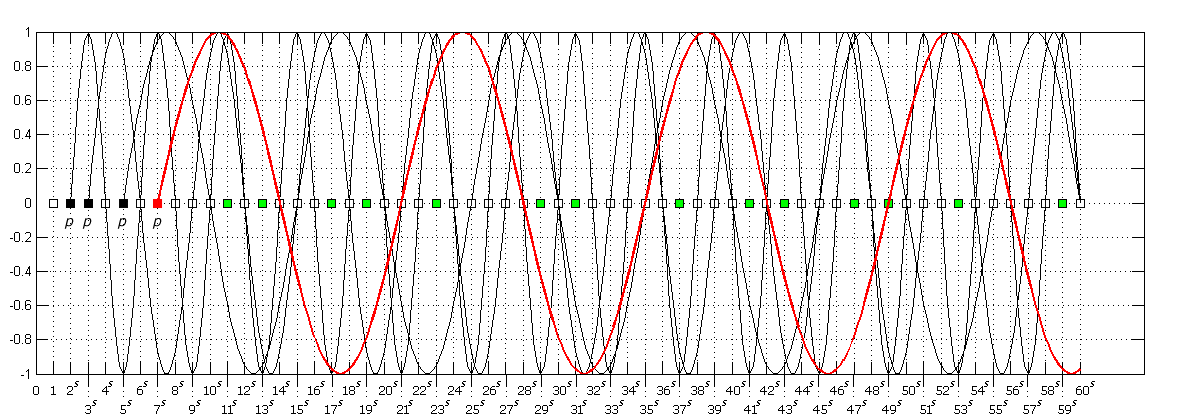}
\end{figure}

At time $n^s=9^s$ a zero-cross occur from the function P($s$)$=3^s+ sin \left(\frac{1}{3^s}\right)$ meaning a $composite$ number has been found, as shown in figure \ref{fig_3}.

At time $p^s=11^s$no zero-cross occur. Thus the function P($s$)$=11^s+ sin \left(\frac{1}{11^s}\right)$ start to zero-cross $odd$ and $even$ multiples of $11^s$, as shown in figure \ref{fig_11}. And so on.

Consequently, all P($s$) prime functions for $p^s>2^s$, zero-cross $odd$ and $even$ multiples of $p^s$.

\begin{figure}[htp]
\caption{P($s$)=$ \left[ 2^s+ sin \left(\frac{1}{2^s}\right)\right] + \left[ 3^s+ sin \left(\frac{1}{3^s}\right) \right]+ \left[ 5^s+ sin \left(\frac{1}{5^s}\right) \right]+ \left[ 7^s+ sin \left(\frac{1}{7^s}\right) \right] + \left[ 11^s+ sin \left(\frac{1}{11^s}\right) \right]$ zero-cross multiples of $2^s$, $3^s$, $5^s$, $7^s$ and $11^s$. In red $p^s=11^s$. \label{fig_11}}
\includegraphics[width=1\textwidth]{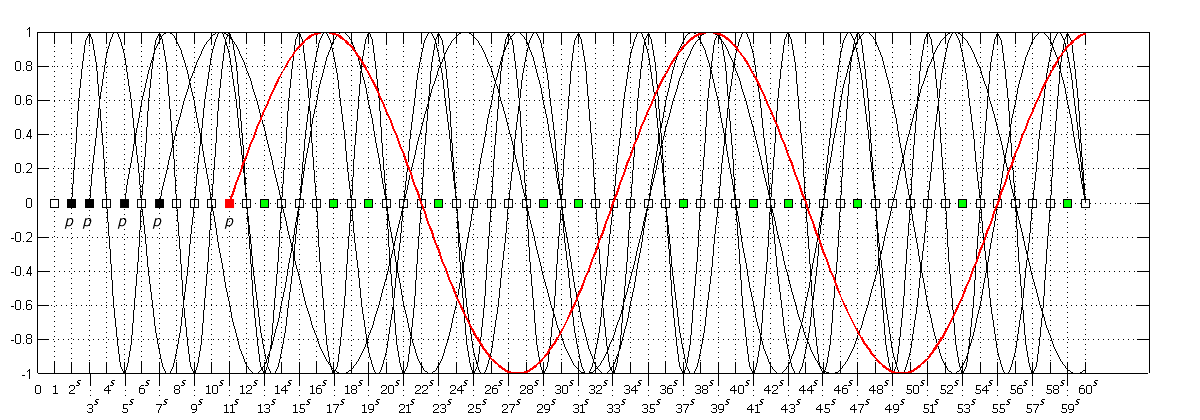}
\end{figure}

\subsection{Dynamical Prime Number Function P($s$) and the Euler product}


The Dynamical Prime Number Function P($s$) expand the primes into its multiples, to extract next primes, by zero-cross (mark) prime multiples as composite numbers similar to the Sieve of Eratosthenes \cite{eras_cite}. Furthermore, the Dynamical Prime Number Function have a close relation to the Euler product.

\begin{proposition}\label{psaseuler}
By representing each time $t$ from the time line as a Riemann zeta function fraction  $\sum_{t=1}^{\infty} \frac{1}{t^s}$

The Dynamical Prime Number function $P(s)$ for $p^s=2^s$ will result in the Riemann's zeta $\frac{1}{2^s} \zeta(s)$ for $s>0$, 
where the zero-cross numbers from $P(s)_{p^s=2^s}$ are the multiples of prime $2^s$.

\begin{equation}
\label{euler2_fig}
P(s)_{p=\{2\}} = \left[ 2^s + sin \left(\frac{1}{2^s}\right) \right] = 
\frac{1}{2^s} \zeta(s) = \frac{1}{2^s} + \frac{1}{4^s} + \frac{1}{6^s} + \frac{1}{8^s} + \frac{1}{10^s} + \dots
\end{equation}

And the numbers from $(1-\frac{1}{2^s}) \zeta(s)$ are possible primes remaining over the time line, as shown in figure \ref{fig_2}.

\begin{equation}
\label{euler2minus1_fig}
\left(1 - \frac{1}{2^s} \right) \zeta(s) = 1 + \frac{1}{3^s} + \frac{1}{5^s} + \frac{1}{7^s} + \frac{1}{9^s} + \frac{1}{11^s} + \frac{1}{13^s} + \frac{1}{15^s} + \dots
\end{equation}

Repeating the process for the following prime number

\begin{equation}
\label{p33xx}
P(s)_{p=\{3\}} = \left[ 3^s + sin \left(\frac{1}{3^s}\right) \right] = \\
\left(\frac{1}{3^s} \right)
\zeta(s) = \frac{1}{3^s} + \frac{1}{6^s} + \frac{1}{9^s} + \frac{1}{12^s} + \frac{1}{15^s} + \frac{1}{18^s} + \dots
\end{equation}

Where by taking out the zero-cross numbers, only possible primes remain in the time line. The $P(s)$ for $p=\{2,3\}$ equals $(1-\frac{1}{3^s})(1-\frac{1}{2^s}) \zeta(s)$ as shown in figure \ref{fig_3}.

\begin{equation}
\label{p33xc}
\left(1 - \frac{1}{3^s} \right)\left(1 - \frac{1}{2^s} \right) \zeta(s) = 1 + \frac{1}{5^s} + \frac{1}{7^s} + \frac{1}{11^s} + \frac{1}{13^s} + \frac{1}{17^s} + \frac{1}{19^s} + \frac{1}{23^s} + \dots
\end{equation}

\end{proposition}
\begin{proof}

The function $P(s)$ will continue decoding primes and zero-cross (encoding) composite numbers, along the time line. 
Thus, the continuation in time of the function $P(s)$ can be defined in terms of the Riemann's zeta function $\zeta(1)$.
Where, for each prime time $p^s$, its continuation in time is the expansion of the prime into its multiples, similar to the $sine$ function when only zero-cross are taken into account, for $s>0$.

\begin{equation}
\label{psequivalent}
 sin \left(\frac{1}{p^s}\right)_{zerocross} = \left(\frac{1}{p^s}\right) \sum_{n=1}^{\infty} \frac{1}{n} 
\end{equation}

\end{proof}

\section{Dynamical Weak Goldbach's Conjecture}

The dynamical prime number function $P(s)$ takes into account all $even$ and $odd$ primes. If taking out of consideration $P(s)$, when $p^s=2^s$, for $s>0$, the zero-cross $composite$ numbers change, excluding only powers of two $\sum_{m>1}^{\infty} 2^{sm}$.

In order to depict the set of $odd$ positive integers bigger than one, $n^s_{odd>1}$, with the set of $odd$ primes  $p^s>2^s$.
It is necessary to modify the dynamical prime number function $P(s)$, for $p^s>2^s$, to zero-cross only $odd$ $composite$ numbers. This is done by doubling the period of the dynamical prime number function $P(s)_{p^s>2^s}$, so the function instead of $odd$ periods $ \left(\frac{1}{p^s}\right)$, consist of $even$ periods $ \left(\frac{1}{2p^s}\right)$. Consequently, the function $P(s)_{p^s>2^s}$ is modified to $P_{odd}(s)$, equation (\ref{eq-psfunction2}), starting at each $odd$ prime time and zero-cross only its $odd$ multiples. While the maximum value of the function exist at $even$ multiples of $odd$ primes.

Furthermore, from the function $P_{odd}(s)$, for $p^s>2^s$, $odd$ multiples of $odd$ primes may be zero-cross by more than one $odd$ prime function, meaning that an $odd$ prime number may be used more than once in the same summation, similar to the method Goldbach's numbers are obtained.

\begin{proof}
Let $s=1$.

Let a zero-cross be represented as a point where the sign of the function changes from positive to negative.

Let a zero-cross over the time line represent a time $t$.

Let times $t$ represent numbers $n$, so $n=t$.

\begin{equation}
\label{eq-psfunction2}
 P_{odd}(s) = \sum_{p^s>2^s} \left[ {p^s} + sin \left(\frac{1}{2p^s}\right) \right]
\end{equation}


Starting at $p=3$, the function $P_{odd}(s)_{p= \{3\}}= \left[ {3} + sin \left(\frac{1}{6}\right) \right]$ zero-crosses $odd$ multiples of $3$, while the maximum value of the function exist at $even$ multiples of $3$, as shown in figure \ref{fig_podd3}.


\begin{figure}[htp]
\caption{$P_{odd}(s)_{p=\{3\}}$ start at $p=3$ and zero-cross $odd$ multiples of 3. In blue $p=3$. \label{fig_podd3}}
\includegraphics[width=1\textwidth]{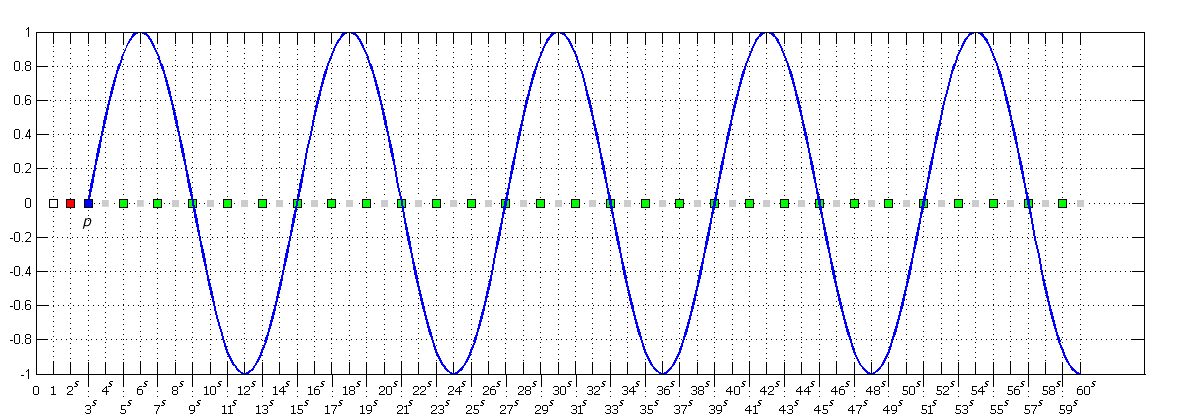}
\end{figure}

At time $p=5$ no zero-cross occur. Therefore, the function $P_{odd}(s)_{p= \{5\}} = \left[ {5} + sin \left(\frac{1}{10}\right) \right]$ start to zero-cross $odd$ multiples of 5, as shown in figure \ref{fig_podd35}.


\begin{figure}[htp]
\caption{$P_{odd}(s)_{p= \{3,5\}} = \left[ {3} + sin \left(\frac{1}{6}\right)\right] + \left[ {5} + sin \left(\frac{1}{10}\right) \right]$ zero-cross $odd$ multiples of $p=3$ and $p=5$. In black $p=3$, in blue $p=5$. \label{fig_podd35}}
\includegraphics[width=1\textwidth]{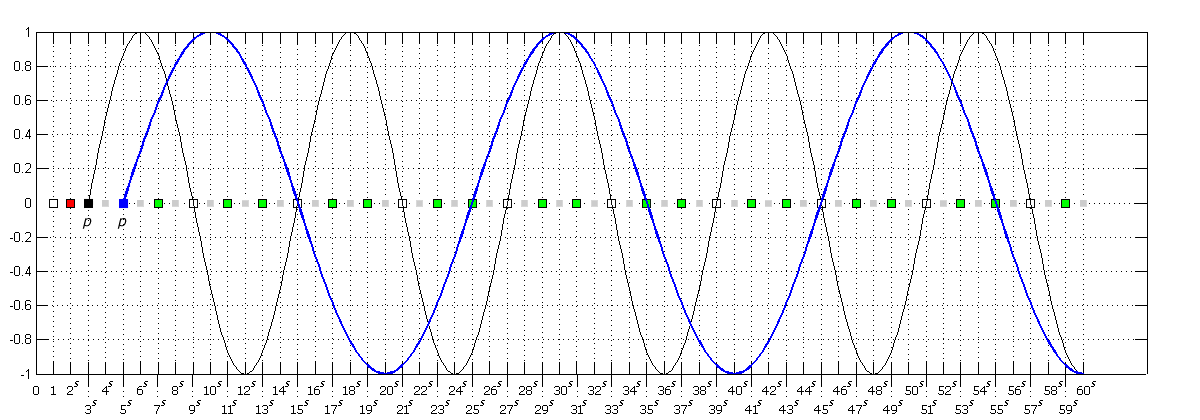}
\end{figure}

At time $p=7$ no zero-cross occur. Thus, the function $P_{odd}(s)_{p= \{7\}} = \left[ {7} + sin \left(\frac{1}{14}\right) \right]$ start to zero-cross $odd$ multiples of 7, as shown in figure \ref{fig_podd357}.


\begin{figure}[htp]
\caption{$P_{odd}(s)_{p= \{3,5,7\}} = \left[ {3} + sin \left(\frac{1}{6}\right)\right] + \left[ {5} + sin \left(\frac{1}{10}\right) \right] + \left[ {7} + sin \left(\frac{1}{14}\right) \right]$ zero-cross $odd$ multiples of $p=3$, $p=5$ and $p=7$. In black $p=3$ and $p=5$, in blue  $p=7$. \label{fig_podd357}}
\includegraphics[width=1\textwidth]{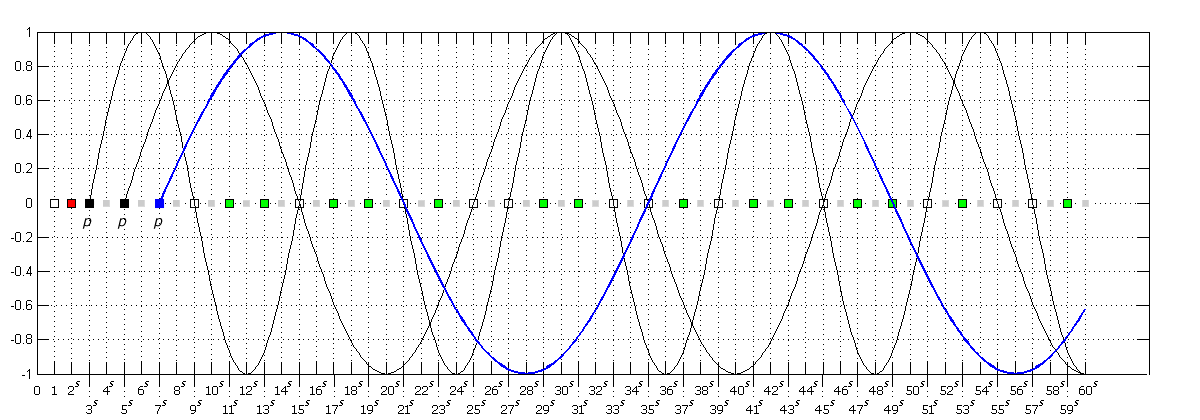}
\end{figure}

%
At time $n=9$ a zero-cross occur from the function $P_{odd}(s)_{p= \{3\}} $ meaning the number is $composite$, as shown in figure \ref{fig_podd3}. 

\begin{figure}[htp]
\caption{$P_{odd}(s)_{p= \{3,5,7,11\}} = \left[ {3} + sin \left(\frac{1}{6}\right)\right] + \left[ {5} + sin \left(\frac{1}{10}\right) \right] + \left[ {7} + sin \left(\frac{1}{14}\right) \right]+ \left[ {11} + sin \left(\frac{1}{22}\right) \right]$ zero-cross $odd$ multiples of $p=3$, $p=5$, $p=7$ and $p=11$. In black $p=3$,  $p=5$,  $p=7$, in blue $p=11$. \label{fig_podd35711}}
\includegraphics[width=1\textwidth]{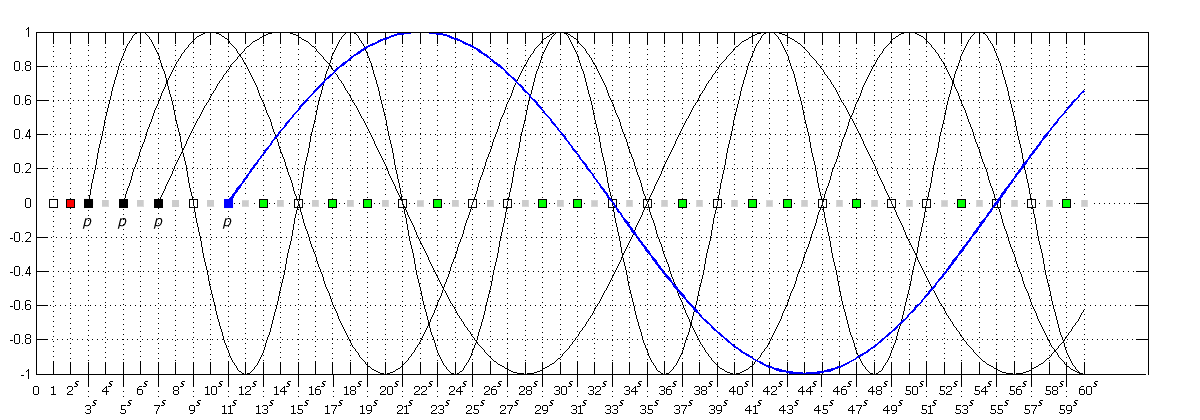}
\end{figure}

At time $p=11$ no zero-cross occur. Thus, the function $P_{odd}(s)_{p= \{11\}} = \left[ {11} + sin \left(\frac{1}{22}\right) \right]$ start to zero-cross $odd$ multiples of 11, as shown in figure \ref{fig_podd35711}.

Consequently, the modified dynamical prime number function $P_{odd}(s)$ for $p>2$, zero-cross all $odd$ composite numbers, starting at $odd$ primes, as shown in figure \ref{fig_podds}. 

\begin{figure}[htp]
\caption{$P(s)_{odd}$, $p>2$; odd multiples of $p>2$ are zero-cross. In red $p=2$, in black $p>2$. \label{fig_podds}}
\includegraphics[width=1\textwidth]{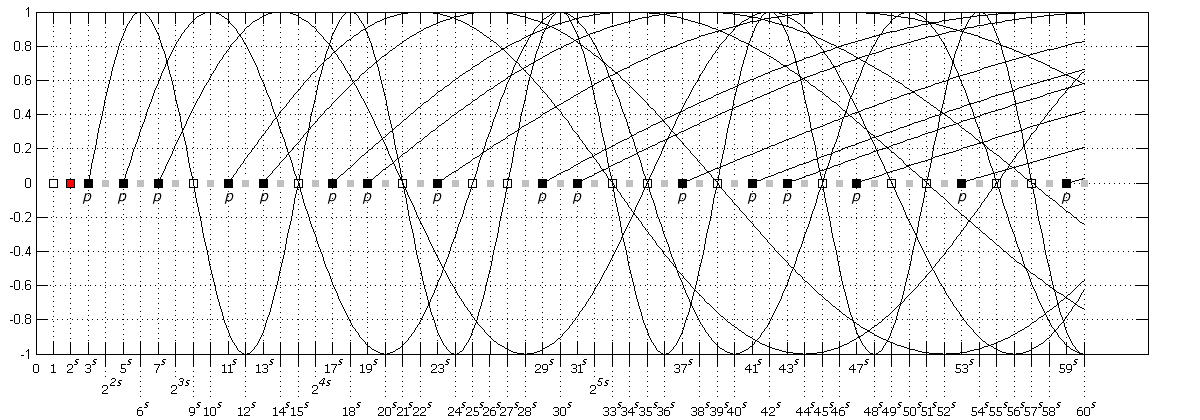}
\end{figure}

%
%
%

It was shown previously that the dynamical prime number function P(s) will zero-cross the entire set of $composite$ numbers, starting at each prime time $p^s$. Moreover, by only taking into account the odd primes, $p^s>2^s$, the zero-cross numbers exclude only powers of two, $2^m$, for $m>1$, from the set $composite$ numbers. Nevertheless, to only zero-cross $odd$ times with the dynamical prime number function $P(s)$, equation (\ref{psfunction}), the period of the function must be doubled as stated in the modified dynamical prime number function $P_{odd}(s)$, equation (\ref{eq-psfunction2}).

In order to demonstrate the weak Goldbach conjecture, first, let's modify the conjecture 
as a dynamical system, where a zero-cross indicates a time $t$, the conjecture states: "Every $odd$ time $t$ greater than 7 will be zero-cross by the sum of three odd dynamical prime number functions $P_{odd}(s)$". (An odd dynamical prime number function $P_{odd}(s)$ can be used more than once in the same summation).

Thus, if the function $P_{odd}(s)$ for $p^s>2^s$, when $s=1$, is able to zero-cross all $odd$ $composite$ numbers from $odd$ primes, then $P_{odd}(s)$ is able depict $odd$ numbers bigger than one $n_{odd > 1}$.

\begin{equation}
\label{weakgold}
 \left(\sum_{p>2}^{\infty} P_{odd}(s) + \sum_{p>2}^{\infty} P_{odd}(s) + \sum_{p>2}^{\infty} P_{odd}(s)\right)  \in   n_{odd > 7}
\end{equation}

Due that times $t$ represent numbers $n$, and the dynamical prime number function $P(s)$ represent prime numbers $p^s$ and its continuation in time, $sin(\frac{1}{p^s})$ for $P(s)$ and $sin(\frac{1}{2p^s})$ for $P_{odd}(s)$. 
If taking out of the equation the continuation in time, $P(s)$ and $P_{odd}(s)$ are represented as prime numbers.



\begin{equation}
\label{psfunction1}
 P(s) = \sum_{p \, \mathrm{ prime}} \left[ {p^s} \right]
\end{equation}

\begin{equation}
\label{psfunction2o}
 P_{odd}(s) = \sum_{p>2} \left[ {p^s} \right]
\end{equation}

The weak Goldbach's conjecture states: "Every odd number $n$ greater than 7 can be expressed as the sum of three odd primes $p>2$". (An odd prime $p>2$ may be used more than once in the same sum.)

\begin{equation}
\label{weakgold}
 \left(\sum_{p>2}^{\infty} p + \sum_{p>2}^{\infty} p + \sum_{p>2}^{\infty} p\right)  \in   n_{odd > 7}
\end{equation}

\end{proof}


\begin{thebibliography}{9}

\bibitem{sieves} 
O'Neill, M. E. 
"The Genuine Sieve of Eratosthenes", 
Journal of Functional Programming, 
Published online by Cambridge University Press 2008,
pp. 10-11, 
DOI:10.1017/S0956796808007004.


\bibitem{eulerproduct1_cite}
Edwards, H. M. 
"The Euler Product Formula",  
New York: Dover 2001, pp. 6-7.

\bibitem{eulerproduct2_cite}
Shimura, G. 
"Euler Products and Eisenstein Series", 
Providence, RI: Amer. Math. Soc., 1997.

\bibitem{mateos1_cite} 
Mateos L.A. "Chaotic Nonlinear Prime Number Function", 
AIP Conf. Proc. CMLS 2011, 1371, pp. 161-170, 
DOI:10.1063/1.3596639.

\bibitem{eras_cite} 
Merritt, D. "Sieve Of Eratosthenes". (December 14, 2008). ÒSieve Of EratosthenesÓ. http://c2.com/cgi/wiki?SieveOfEratosthenes. Retrieved on 2012-02-15.











\end{thebibliography}
\end{document}